\documentclass[10pt,twoside]{amsart}
\usepackage{amsmath,amsfonts,amssymb,amsxtra,setspace,xspace,graphicx,lmodern,psfrag,epsfig,color,latexsym,mathrsfs}
\usepackage[colorlinks=true]{hyperref}\hypersetup{urlcolor=blue, citecolor=red}
\usepackage{pgfplots}

\newtheorem{thm}{Theorem}
\newtheorem{lem}[thm]{Lemma}
\newtheorem{cor}[thm]{Corollary}
\newtheorem{prop}[thm]{Proposition}

\newcommand{\arxiv}[1]{\href{https://arxiv.org/abs/#1}{arXiv: #1}}

\newcommand{\R}{{\mathbb R}}
\renewcommand{\S}{{\mathbb S}}
\newcommand{\N}{{\mathbb N}}
\newcommand{\be}[1]{\begin{equation}\label{#1}}
\newcommand{\ee}{\end{equation}}
\renewcommand{\(}{\left(}
\renewcommand{\)}{\right)}
\newcommand{\nrm}[2]{\|{#1}\|_{\mathrm L^{#2}(\S^n)}}
\newcommand{\nrmr}[2]{\left\|{#1}\right\|_{\mathrm L^{#2}(\R^n)}}
\newcommand{\nrmrs}[2]{\left\|{#1}\right\|_{\mathrm L_\star^{#2}(\R^n)}}
\newcommand{\isd}[1]{\int_{\S^n}{#1}\,d\mu}
\newcommand{\ird}[1]{\int_{\R^n}{#1}\,dx}

\begin{document}
\title[Interpolation inequalities and fractional operators on the sphere]{Optimal functional inequalities for fractional operators on the sphere and applications}

\author{Jean Dolbeault}
\address{[J.~Dolbeault] Ceremade, UMR CNRS n$^\circ$~7534, Universit\'e Paris-Dauphine, PSL research university, Place de Lattre de Tassigny, 75775 Paris 16, France.}
\email{\href{mailto:dolbeaul@ceremade.dauphine.fr}{dolbeaul@ceremade.dauphine.fr}}
\author{An Zhang}
\address{[A.~Zhang] Ceremade, UMR CNRS n$^\circ$~7534, Universit\'e Paris-Dauphine, PSL research university, Place de Lattre de Tassigny, 75775 Paris 16, France.}
\email{\href{mailto:zhang@ceremade.dauphine.fr}{zhang@ceremade.dauphine.fr}}
\date{\today}
\begin{abstract} This paper is devoted to the family of optimal functional inequalities on the $n$-dimensional sphere $\S^n$, namely
\[
\frac{\nrm Fq^2-\nrm F2^2}{q-2}\le\mathsf C_{q,s}\isd{F\,\mathcal L_s\kern0.5pt F}\quad\forall\,F\in\mathrm H^{s/2}(\S^n)
\]
where $\mathcal L_s$ denotes a fractional Laplace operator of order $s\in(0,n)$, $q\in[1,2)\cup(2,q_\star]$, $q_\star=2\,n/(n-s)$ is a critical exponent and $d\mu$ is the uniform probability measure on $\S^n$. These inequalities are established with optimal constants using spectral properties of fractional operators. Their consequences for fractional heat flows are considered. If $q>2$, these inequalities interpolate between fractional Sobolev and subcritical fractional logarithmic Sobolev inequalities, which correspond to the limit case as $q\to2$. For $q<2$, the inequalities interpolate between fractional logarithmic Sobolev and fractional Poincar\'e inequalities. In the subcritical range $q<q_\star$, the method also provides us with remainder terms which can be considered as an improved version of the optimal inequalities. The case $s\in(-n,0)$ is also considered. Finally, weighted inequalities involving the fractional Laplacian are obtained in the Euclidean space, using a stereographic projection.

\end{abstract}
\maketitle
\thispagestyle{empty}
\vspace*{-0.8cm}

\section{Introduction and main results}\label{Sec:Intro}

Let us consider the unit sphere $\S^n$ with $n\ge1$ and assume that the measure $d\mu$ is the uniform probability measure, which is also the measure induced on $\S^n$ by Lebesgue's measure on $\R^{n+1}$, up to a normalization constant. With $\lambda\in(0,n)$, $p=\frac{2\,n}{2\,n-\lambda}\in(1,2)$ or equivalently $\lambda=\frac{2\,n}{p'}$ where $\frac1p+\frac1{p'}=1$, according to~\cite{MR717827}, the sharp \emph{Hardy-Littlewood-Sobolev inequality on~$\S^n$} reads
\be{HLS}
\iint_{\S^n\times\S^n}F(\zeta)\,|\zeta-\eta|^{-\lambda}\,F(\eta)\,d\mu(\zeta)\,d\mu(\eta)\le\frac{\Gamma(n)\,\Gamma\big(\tfrac {n-\lambda}2\big)}{2^\lambda\,\Gamma\big(\tfrac n2\big)\,\Gamma\big(\tfrac np\big)}\,\nrm Fp^2\,.
\ee
For the convenience of the reader, the definitions of all parameters, their ranges and their relations have been collected in Appendix~\ref{Appendix:Notations}.

By the Funk-Hecke formula, the left-hand side of the inequality can be written as
\begin{multline}\label{FH}
\iint_{\S^n\times\S^n}F(\zeta)\,|\zeta-\eta|^{-\lambda}\,F(\eta)\,d\mu(\zeta)\,d\mu(\eta)\\
=\frac{\Gamma(n)\,\Gamma\big(\tfrac {n-\lambda}2\big)}{2^\lambda\,\Gamma\big(\tfrac n2\big)\,\Gamma\big(\tfrac np\big)}\,
\sum_{k=0}^\infty\,\frac{\Gamma(\tfrac np)\,\Gamma(\tfrac n{p'}+k)}{\Gamma(\tfrac n{p'})\,\Gamma(\tfrac np+k)}\isd{|F_{(k)}|^2}
\end{multline}
where $F=\sum_{k=0}^\infty F_{(k)}$ is a decomposition on spherical harmonics, so that $F_{(k)}$ is a spherical harmonic function of degree $k$. See~\cite[Section~4]{MR2848628} for details on the computations and, \emph{e.g.},~\cite{MR0199449} for further related results. With the above representation, inequality~\eqref{HLS} is equivalent to
\be{HLS2}
\sum_{k=0}^\infty\,\frac{\Gamma(\tfrac np)\,\Gamma(\tfrac n{p'}+k)}{\Gamma(\tfrac n{p'})\,\Gamma(\tfrac np+k)}\isd{|F_{(k)}|^2}\le\nrm Fp^2\,.
\ee
By duality, with $q_\star=q_\star(s)$ defined by
\be{CriticalExponent}
q_\star=\frac{2\,n}{n-s}
\ee
or equivalently $s=n\,(1-2/q_\star)$, we obtain the \emph{fractional Sobolev inequality on~$\S^n$}
\be{Sobolev}
\nrm F{q_\star}^2\le\isd{F\,\mathcal K_s\kern0.5pt F}\quad\forall\,F\in\mathrm H^{s/2}(\S^n)
\ee
for any $s\in(0,n)$, where
\be{gamma}
\isd{F\,\mathcal K_s\kern0.5pt F}:=\sum_{k=0}^\infty\gamma_k\big(\tfrac n{q_\star}\big)\isd{|F_{(k)}|^2}
\ee
and
\[
\gamma_k(x):=\frac{\Gamma(x)\,\Gamma(n-x+k)}{\Gamma(n-x)\,\Gamma(x+k)}\,.
\]
With $s\in(0,n)$, the exponent $q_\star$ is in the range $(2,\infty)$. Inequalities~\eqref{HLS} and~\eqref{Sobolev} are related by $q_\star=p'$ so that
\[
p=\frac{2\,n}{n+s}\quad\mbox{and}\quad\lambda=n-s\,.
\]
We shall refer to $q=q_\star(s)$ given by~\eqref{CriticalExponent} as the \emph{critical case} and our purpose is to study the whole range of the \emph{subcritical interpolation inequalities}
\be{interpolation}
\frac{\nrm Fq^2-\nrm F2^2}{q-2}\le\mathsf C_{q,s}\isd{F\,\mathcal L_s\kern0.5pt F}\quad\forall\,F\in\mathrm H^{s/2}(\S^n)
\ee
for any $q\in[1,2)\cup(2,q_\star]$, where
\[
\mathcal L_s\kern0.5pt:=\frac1{\kappa_{n,s}}\(\mathcal K_s-\mathrm{Id}\)\quad\mbox{with}\quad\kappa_{n,s}:=\frac{\Gamma\big(\frac n{q_\star}\big)}{\Gamma\big(n-\frac n{q_\star}\big)}=\frac{\Gamma\big(\frac{n-s}2\big)}{\Gamma\big(\frac{n+s}2\big)}\,.
\]

If $q=q_\star$,~inequalities \eqref{Sobolev} and~\eqref{interpolation} are identical, the optimal constant in~\eqref{interpolation} is $\mathsf C_{q_\star,s}=\frac{\kappa_{n,s}}{q_\star-2}$, and we recall that~\eqref{Sobolev} is equivalent to the fractional Sobolev inequality on the Euclidean space (see the proof of Theorem~\ref{Thm:Main2} in Section~\ref{Sec:Euclidean} for details). The usual conformal fractional Laplacian is defined~by
\[
\mathcal A_s:=\frac1{\kappa_{n,s}}\,\mathcal K_s=\mathcal L_s+\frac1{\kappa_{n,s}}\,\mathrm{Id}\,.
\]
For brevity, we shall say that $\mathcal L_s$ is the \emph{fractional Laplacian} of order $s$, or simply the \emph{fractional Laplacian}.

We observe that $\gamma_0(n/q)-1=0$ and $\gamma_1(n/q)-1=q-2$. A straightforward computation gives
\[
\isd{F\,\mathcal L_s\kern0.5pt F}:=\sum_{k=1}^\infty\delta_k\big(\tfrac n{q_\star}\big)\isd{|F_{(k)}|^2}
\]
where the spectrum of $\mathcal L_s$ is given by
\[
\delta_k(x):=\frac{\Gamma(n-x+k)}{\Gamma(x+k)}-\frac{\Gamma(n-x)}{\Gamma(x)}\,.
\]
The case corresponding to $s=2$ and $n\ge3$, where
\[
\frac1{\kappa_{n,2}}=\frac14\,n\,(n-2), \quad \mathcal L_2=-\Delta,\quad \mathcal A_2=-\Delta+\frac14\,n\,(n-2)
\]
and $\Delta$ stands for the \emph{Laplace-Beltrami operator} on $\S^n$, has been considered by W.~Beckner: in~\cite[page~233,~(35)]{MR1230930} he observed that
\[
\delta_k\big(\tfrac nq\big)\le\delta_k\big(\tfrac n{q_\star}\big)=k\,(k+n-1)
\]
if $q\in(2,q_\star(2)]$, where $q_\star=q_\star(2)=2\,n/(n-2)$ and $(k\,(k+n-1))_{k\in\N}$ is the sequence of the eigenvalues of $-\Delta$ according to, \emph{e.g.},~\cite{MR0282313}. This establishes the interpolation inequality
\be{NonFractional}
\nrm Fq^2-\nrm F2^2\le\frac{q-2}n\,\nrm{\nabla F}2^2\quad\forall\,F\in\mathrm H^1(\S^n)
\ee
where $\mathsf C_{q,2}=1/n$ is the optimal constant: see \cite[(35),~Theorem~4]{MR1230930} for details. An earlier proof of the inequality with optimal constant can be found in \cite[Corollary~6.2]{BV-V}, with a proof based on \emph{rigidity} results for elliptic partial differential equations. Our main result generalizes the interpolation inequalities~\eqref{NonFractional} to the case of the fractional operators $\mathcal L_s$, and relies on W.~Beckner's approach. In particular, as in~\cite{MR1230930}, we characterize the optimal constant $\mathsf C_{q,s}$ in~\eqref{interpolation} using a spectral gap property.

After dividing both sides of~\eqref{NonFractional} by $(q-2)$ we obtain an inequality which, for $s=2$, also makes sense for any $q\in[1,2)$. When $q=1$, this is actually a variant of the Poincar\'e inequality (or, to be precise, the Poincar\'e inequality written for $|F|$), and the range $q>1$ has been studied using the \emph{carr\'e du champ} method, also known as the $\Gamma_2$ calculus, by D.~Bakry and M.~Emery in~\cite{MR889476}. Actually their method covers the range corresponding to $1\le q<\infty$ if $n=1$ and
\[1\le q\le2^\#:=\frac{2\,n^2+1}{(n-1)^2}\quad \text{if} ~\, n\ge2.\]
In the special case $q=2$, the left-hand side of \eqref{NonFractional} has to be replaced by the entropy
\[\isd{F^2\,\log\(\frac{F^2}{\nrm F2^2}\)}\,.\] Still under the condition that $s=2$, the whole range $1\le q<\infty$ when $n=2$, and $1\le q\le 2\,n/(n-2)$ if $n\ge3$ can be covered using nonlinear flows as shown in~\cite{MR2381156,MR3229793,dolbeault:hal-01206975}.

\medskip All these considerations motivate our first result, which generalizes known results for $\mathcal L_2=-\Delta$ to the case of the \emph{fractional Laplacian}~$\mathcal L_s$.
\begin{thm}\label{Thm:Main1} Let $n\ge1$, $s\in(0,n]$, $q\in[1,2)\cup(2,q_\star]$, with $q_\star$ given by \eqref{CriticalExponent}, if $s<n$, and $q\in[1,2)\cup(2,\infty)$ if $s=n$. Inequality~\eqref{interpolation} holds with sharp constant
\[
\mathsf C_{q,s}=\frac{n-s}{2\,s}\,\frac{\Gamma\big(\frac{n-s}2\big)}{\Gamma\big(\frac{n+s}2\big)}\,.
\]
\end{thm}
With our previous notations, this amounts to $\mathsf C_{q,s}=\frac{\kappa_{n,s}}{q_\star-2}=\frac{n-s}{2\,s}\,\kappa_{n,s}$. Remarkably, $\mathsf C_{q,s}$ is independent of $q$. Equality in~\eqref{interpolation} is achieved by constant functions. The issue of the optimality of $\mathsf C_{q,s}$ is henceforth somewhat subtle. If we define the functional
\be{Q}
\mathcal Q[F]:=\frac{(q-2)\isd{F\,\mathcal L_s\kern0.5pt F}}{\nrm Fq^2-\nrm F2^2}
\ee
on the subset $\mathscr H^{s/2}$ of the functions in $\mathrm H^{s/2}(\S^n)$ which are not almost everywhere constant, then $\mathsf C_{q,s}$ can be characterized by
\[
\mathsf C_{q,s}^{-1}=\inf_{F\in\mathscr H^{s/2}}\mathcal Q[F]\,.
\]
This minimization problem will be discussed in Section~\ref{Sec:Conclusion}.

Our key estimate is a simple convexity observation that is stated in Lemma~\ref{Lem:MonotonicityGamma}. The optimality in~\eqref{interpolation} is obtained by performing a linearization, which corresponds to an asymptotic regime as we shall see in Section~\ref{SubSec:Poincare}. Technically, this is the reason why we are able to identify the optimal constant. The asymptotic regime can be investigated using a flow. Indeed, a first consequence of Theorem~\ref{Thm:Main1} is that we may apply entropy methods to the generalized fractional heat flow
\be{FHeatFlow}
\frac{\partial u}{\partial t}-q\,\nabla\cdot\(u^{1-\frac1q}\,\nabla(-\Delta)^{-1}\,\mathcal L_s\kern0.5pt u^\frac1q\)=0\,.
\ee
Notice that~\eqref{FHeatFlow} is a $1$-homogeneous equation, but that it is nonlinear when $q\neq1$ and $s\neq2$. Let us define a \emph{generalized entropy} by
\[
\mathcal E_q[u]:=\frac1{q-2}\left[\;\(\isd u\)^\frac2q-\isd{u^\frac2q}\right]\,.
\]
It is straightforward to check that for any positive solution to~\eqref{FHeatFlow} which is smooth enough and has sufficient decay properties as $|x|\to+\infty$, we have
\[
\frac d{dt}\mathcal E_q[u(t,\cdot)]=-\,2\isd{\nabla u^\frac1q\cdot\nabla(-\Delta)^{-1}\,\mathcal L_s\kern0.5pt u^\frac1q}=-\,2\isd{u^\frac1q\,\mathcal L_s\kern0.5pt u^\frac1q}\,,
\]
so that by applying~\eqref{interpolation} to $F=u^{1/q}$ we obtain the exponential decay of $\mathcal E_q[u(t,\cdot)]$.
\begin{cor}\label{Cor:FHeatFlow} Let $n\ge1$, $s\in(0,n]$, $q\in[1,2)\cup(2,q_\star]$ if $s<n$, with $q_\star$ given by \eqref{CriticalExponent}, and $q\in[1,2)\cup(2,\infty)$ if $s=n$. If $u$ is a positive function in $C^1(\R^+;\mathrm L^\infty(\S^n))$ such that $u^{1/q}\in C^1(\R^+;\mathrm H^{s/2}(\S^n))$ and if $u$ solves~\eqref{FHeatFlow} on $\S^n$ with initial datum~$u_0>0$, then
\[
\mathcal E_q[u(t,\cdot)]\le\mathcal E_q[u_0]\,e^{-\,2\,\mathsf C_{q,s}^{-1}\,t} \quad\forall\,t\ge0\,.
\]\end{cor}
The exponential rate is determined by the asymptotic regime as $t\to+\infty$. The value of \emph{the optimal constant} $\mathsf C_{q,s}$ is indeed determined by the \emph{spectral gap of the linearized problem} around non-zero constant functions. From the expression of~\eqref{FHeatFlow}, which is not even a linear equation whenever $s\neq2$, we observe that the interplay of optimal fractional inequalities and fractional diffusion flows is not straightforward, while for $s=2$, the generalized entropy $\mathcal E_q$ enters in the framework of the so-called $\varphi$-entropies and is well understood in terms of gradient flows: see for instance~\cite{MR1842428,MR2081075,MR2448650}. When $s=2$, it is also known from~\cite{MR889476} that heat flows can be used in the framework of the \emph{carr\'e du champ} method to establish the inequalities at least for exponents in the range $q\le2^\#$ if $n\ge2$, and that the whole subcritical range of exponents can be covered using nonlinear diffusions as in~\cite{MR2381156,MR3229793,dolbeault:hal-01206975} (and also the critical exponent if $n\ge3$). Even better, \emph{rigidity} results, that is, uniqueness of positive solutions (which are therefore constant functions) follow by this technique. So far there is no analogue in the case of fractional operators, except for one example found in~\cite{MR3279352} when $n=1$.

When $s=2$, the \emph{carr\'e du champ} method provides us with an integral remainder term and, as a consequence, with an improved version of~\eqref{interpolation}. As we shall see, our proof of~Theorem~\ref{Thm:Main1} establishes another improved inequality, by construction: see Corollary~\ref{Cor:improvedInterpolation}. This also suggests another direction, which is more connected with the duality that relates~\eqref{HLS} and~\eqref{Sobolev}. Let us describe the main idea. The operator $\mathcal K_s$ is positive definite and we can henceforth consider $\mathcal K_s^{1/2}$ and $\mathcal K_s^{-1}$. Moreover, using~\eqref{FH} and~\eqref{gamma}, we know that
\[
\iint_{\S^n\times\S^n}G(\zeta)\,|\zeta-\eta|^{-\lambda}\,G(\eta)\,d\mu(\zeta)\,d\mu(\eta)=\frac{\Gamma(n)\,\Gamma(\frac s2)}{2^\lambda\,\Gamma(\tfrac n2)\,\Gamma(n+\frac s2)}\isd{G\,\mathcal K_s^{-1}\,G}\,.
\]
Expanding the square \[\isd{\big|\mathcal K_s^{1/2}\kern0.5pt F-\mathcal K_s^{-1/2}\kern0.5pt G\big|^2}\] with $G=F^{q_\star-1}$ so that $F\,G=F^{q_\star}=G^p$ where $q_\star$ and $p$ are H\"older conjugates, we get a comparison of the difference of the two terms which show up in~\eqref{HLS} and~\eqref{Sobolev} and, as a result, an \emph{improved fractional Sobolev inequality on~$\S^n$}. The reader interested in the details of the proof is invited to consult~\cite{MR3227280} for a similar result.
\begin{prop}\label{Prop:Square} Let $n\ge 1$ and $s\in(0,n)$. Consider $q_\star$ given by~\eqref{CriticalExponent}, $p=q_\star'=\frac{2\,n}{n+s}$ and $\lambda=n-s$. For any $F\in\mathrm H^{s/2}(\S^n)$, if $G=F^{q_\star-1}$, then
\begin{align*}
\nrm Gp^2-2^{\lambda}\,\frac{\Gamma(\frac n2)\,\Gamma(n+\tfrac s2)}{\Gamma(n)\,\Gamma(\tfrac s2)}\,
\iint_{\S^n\times\S^n}G(\zeta)\,|\zeta-\eta|^{-\lambda}\,G(\eta)\,d\mu(\zeta)\,d\mu(\eta)\\
\le\nrm F{q_\star}^{2\kern0.5pt(q_\star-2)}\,\(\isd{F\,\mathcal K_s\kern0.5pt F}-\nrm F{q_\star}^2\)\,.
\end{align*}
\end{prop}
Still in the critical case $q=q_\star$, by using the \emph{fractional Yamabe flow} and taking inspiration from~\cite{1101,MR3227280,MR3276166,2014arXiv1404.1028J,MR3429269}, it is possible to give improvements of the above inequality and in particular improve on the constant which relates the left- and the right-hand sides of the inequality in Proposition~\ref{Prop:Square}. We will not go further in this direction because of the delicate regularity properties of the fractional Yamabe flow and because, so far the method does not allow to characterize the best constant in the improvement. Let us mention that, in the critical case $q=q_\star$, further estimates of Bianchi-Egnell type have also been obtained in~\cite{MR3179693,MR3429269} for fractional operators. In this paper, we shall rather focus on the subcritical range. It is however clear that there is still space for further improvements, or alternative proofs of~\eqref{Sobolev} which rely neither on rearrangements as in~\cite{MR717827} nor on inversion symmetry as in~\cite{MR2659680,MR2858468,MR2848628}, for the simple reason that our method fails to provide us with a proof of the Bianchi-Egnell estimates in the critical case.

For completeness let us quote a few other related results. Symmetrization techniques and the method of competing symmetries are both very useful to identify the optimal functions: the interested reader is invited to refer to~\cite{MR1817225} and~\cite{MR1038450}, respectively, when $s=2$. In this paper, we shall use notations inspired by~\cite{MR1230930}, but at this point it is worth mentioning that in~\cite{MR1230930} the emphasis is put on logarithmic Hardy-Littlewood-Sobolev inequalities and their dual counterparts, which are $n$-dimensional versions of the Moser-Trudinger-Onofri inequalities. Some of these results were obtained simultaneously in~\cite{MR1143664} with some additional insight on optimal functions gained from rearrangements and from the method of competing symmetries. Concerning observations on duality, we refer to the introduction of~\cite{MR1143664}, which clearly refers the earlier contributions of various authors in this area. For more recent considerations on $n$-dimensional Moser-Trudinger-Onofri inequalities, see, \emph{e.g.},~\cite{delPino26062012}.

Section~\ref{Sec:Interpolation} is devoted to the proof of Theorem~\ref{Thm:Main1}. As already said, we shall take advantage of the subcritical range to obtain remainder terms and improved inequalities. Improvements in the subcritical range have been obtained in the case of non-fractional interpolation inequalities in the context of fast diffusion equations in~\cite{Dolbeault2013917,Dolbeault15052015}. In this paper we shall simply take into account the terms which appear by difference in the proof of Theorem~\ref{Thm:Main1}: see Corollary~\ref{Cor:improvedInterpolation} in Section~\ref{SubSec:Remainder}. Although this approach does not provide us with an alternative proof of the optimality of the constant $\mathsf C_{q,s}$ in~\eqref{interpolation}, variational methods will be applied in Section~\ref{Sec:Conclusion} in order to explain \emph{a posteriori} why the value of the optimal value of $\mathsf C_{q,s}$ is determined by the spectral gap of a linearized problem. Some useful information on the spectrum of~$\mathcal L_s$ is detailed in Appendix~\ref{Appendix:Spectrum}.

\medskip Our next result is devoted to the singular case of inequality~\eqref{interpolation} corresponding to the limit as $q=2$. We establish a family of \emph{sharp fractional logarithmic Sobolev inequalities}, in the subcritical range.
\begin{cor}\label{corlog} Let $s\in(0,n]$. Then we have the \emph{sharp logarithmic Sobolev inequality}
\be{logsob}
\isd{|F|^2\,\log\(\frac{|F|}{\nrm F2}\)}\le\mathsf C_{2,s}\isd{F\,\mathcal{L}_s\kern0.5pt F}\quad\forall\,F\in\mathrm H^{s/2}(\S^n)\,.
\ee
Equality is achieved only by constant functions and $\mathsf C_{2,s}=\frac{n-s}{2\,s}\kern1pt\kappa_{n,s}$ is optimal.\end{cor}
This result completes the picture of Theorem~\ref{Thm:Main1} and shows that, under appropriate precautions, the case $q=2$ can be put in a common picture with the cases corresponding to \hbox{$q\neq2$}. Taking the limit as $s\to0_+$, we recover Beckner's fractional logarithmic Sobolev inequality as stated in~\cite{MR1164616,MR1441924}. In that case, $q=2$ is critical, from the point of view of the fractional operator. The proof of Corollary~\ref{corlog} and further considerations on the $s=0$ limit will be given in Section~\ref{Sec:LogSob}.

\medskip Definition~\eqref{gamma} of $\mathcal K_s$ also applies to the range $s\in(-n,0)$ and the reader is invited to check that
\[
\mathcal K_s^{-1}=\mathcal K_{-s}\quad\forall\,s\in(0,n)
\]
is defined by the sequence of eigenvalues $\gamma_k(n/p)$ where $p=2\,n/(n+s)$ is the H\"older conjugate of $q_\star(s)$ given by~\eqref{CriticalExponent}. It is then straightforward to check that the sharp \emph{Hardy-Littlewood-Sobolev inequality on~$\S^n$}~(see \eqref{HLS2}) can be written as
\be{HLS3}
\frac{\nrm Fp^2-\nrm F2^2}{p-2}\le\frac{\kappa_{n,-s}}{2-p}\isd{F\,\mathcal L_{-s}\kern0.5pt F}\quad\forall\,F\in\mathrm L^2(\S^n)
\ee
where
\[
p=\frac{2\,n}{n+s}\in(1,2), \quad \mathcal L_{-s}\kern0.5pt:=\frac1{\kappa_{n,-s}}\(\mathrm{Id}-\mathcal K_{-s}\)~ \quad \text{and}\quad \kappa_{n,-s}=\frac{\Gamma\big(\frac{n+s}2\big)}{\Gamma\big(\frac{n-s}2\big)}\,.
\]
Notice that $\kappa_{n,-s}=1/\kappa_{n,s}$. A first consequence is that we can rewrite the result of Proposition~\ref{Prop:Square} as
\[
\nrm Gp^2-\isd{G\,\mathcal K_{-s}\kern0.5pt G}\le\nrm F{q_\star}^{2\kern0.5pt(q_\star-2)}\,\(\isd{F\,\mathcal K_s\kern0.5pt F}-\nrm F{q_\star}^2\)\,.
\]
for any $F\in\mathrm H^{s/2}(\S^n)$ and $G=F^{q_\star-1}$, where $n\ge 1$, $s\in(0,n)$, $q_\star$ is given by~\eqref{CriticalExponent} and $p=q_\star'$. A second consequence of the above observations is the extension of Theorem~\ref{Thm:Main1} to the range $(-n,0)$.
\begin{thm}\label{Thm:Main1bis} Let $n\ge1$, $s\in(-n,0)$ and $q\in[1,2\,n/(n-s))$. Inequality~\eqref{interpolation} holds with $\mathcal L_s:=\kappa_{n,-s}\,(\mathrm {Id}-\mathcal K_s)$ and sharp constant
\[\mathsf C_{q,s}=\frac{n-s}{2\,|s|}\,\frac{\Gamma\big(\frac{n-s}2\big)}{\Gamma\big(\frac{n+s}2\big)}\,.\]\end{thm}
The results of Theorems~\ref{Thm:Main1} and~\ref{Thm:Main1bis} are summarized in Figure ~\ref{Fig2}.
\begin{figure}[ht]
\begin{center}
\includegraphics[height=4cm]{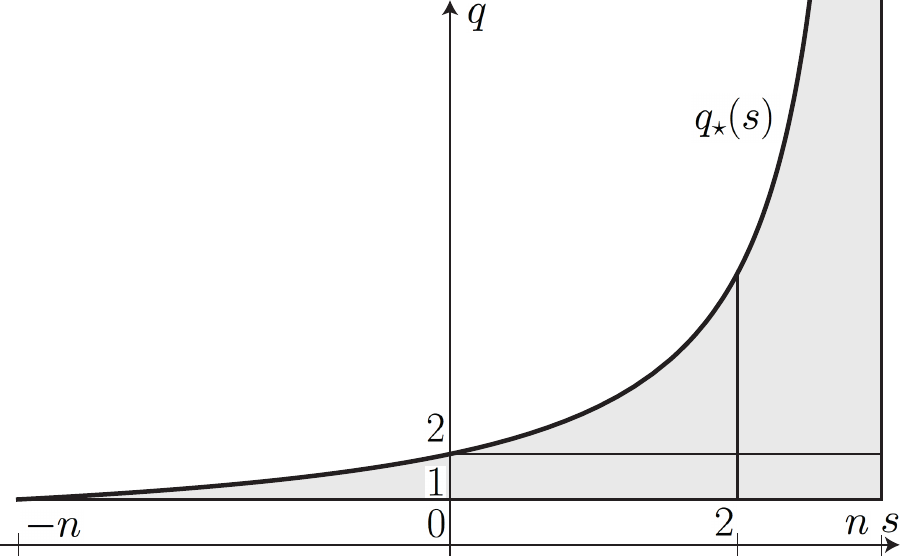}
\caption{\label{Fig2}{\sl\small The optimal constant $\mathsf C_{q,s}$ in~\eqref{interpolation} is independent of $q$ and determined for any given $s$ by the critical case $q=q_\star(s)$ which corresponds to the \emph{Hardy-Littlewood-Sobolev inequality}~\eqref{HLS} if $s\in(-n,0)$ and to the \emph{Sobolev inequality}~\eqref{Sobolev} if $s\in(0,n)$. The case $s=0$ is covered by Corollary~\ref{Cor:s=0}, while $q=2$ corresponds to the \emph{fractional logarithmic Sobolev} inequality~\eqref{LogSob} if $s=0$ and the subcritical \emph{fractional logarithmic Sobolev} inequality by Corollary ~\ref{corlog} if $s\in(0,n]$.}}
\end{center}
\end{figure}

\medskip To conclude with the outline of this paper, Section~\ref{Sec:Euclidean} is devoted to the stereographic projection and consequences for functional inequalities on the Euclidean space. By stereographic projection,~\eqref{Sobolev} becomes
\[
\nrmr f{q_\star}^2\le\mathsf S_{n,s}\,\|f\|_{\dot{\mathrm H}^{s/2}(\R^n)}^2\quad\forall\,f\in\dot{\mathrm H}^{s/2}(\R^n)\,,
\]
where
\[\|f\|_{\dot{\mathrm H}^{s/2}(\R^n)}^2:=\ird{f\,(-\Delta)^{s/2}f}\]
and the optimal constant is such that
\[
\mathsf S_{n,s}=\kappa_{n,s}\,|\S^n|^{\frac2{q_\star}-1}\,.
\]
The fact that~\eqref{Sobolev} is equivalent to the fractional Sobolev inequality on the Euclidean space is specific to the critical exponent $q=q_\star(s)$. In the subcritical range, weights appear. Let us introduce the weighted norm
\[
\nrmrs f{q,\beta}^q:=\ird{|f|^q\,(1+|x|^2)^{-\frac\beta2}}\,.
\]
The next result is inspired by a non-fractional computation done in~\cite{DET} and relies on the stereographic projection.
\begin{thm}\label{Thm:Main2} Let $n\ge1$, $s\in(0,n)$, $q\in(2,q_\star)$ with $q_\star$ given by~\eqref{CriticalExponent} and $\beta=2\,n\,(1-\frac q{q_\star})$. Then we have the weighted inequality
\be{fCKNopt}
\nrmrs f{q,\beta}^2\le\mathsf a\,\|f\|_{\dot{\mathrm H}^{s/2}(\R^n)}^2+\mathsf b\,\nrmrs f{2,2s}^2\quad\forall\,f\in C^\infty_0(\R^n)
\ee
where
\[
\mathsf a=\frac{q-2}{q_\star-2}\,\kappa_{n,s}\,2^{\kern1pt n\kern0.5pt(\frac2{q_\star}-\frac2q)}\,|\S^n|^{\frac 2q-1} \quad \text{and}\quad \mathsf b=\frac{q_\star-q}{q_\star-2}\,2^{\kern1pt n\kern0.5pt(1-\frac 2q)}\,|\S^n|^{\frac 2q-1}\,.
\]
Moreover, if $q<q_\star$, equality holds in~\eqref{fCKNopt} if and only if $f$ is proportional to $f_{s,\star}(x):=(1+|x|^2)^{-\frac{n-s}2}$.
\end{thm}
This result is one of the few examples of optimal functional inequalities involving fractional operators on $\R^n$. It touches the area of fractional Hardy-Sobolev inequalities and weighted fractional Sobolev inequalities, for which we refer to~\cite{MR3366777,chensymmetry} and~\cite{Chen-Yang16}, respectively, and the references therein. The wider family of Caffarelli-Kohn-Nirenberg type inequalities raises additional difficulties, for instance related with symmetry and symmetry breaking issues, which are so far essentially untouched in the framework of fractional operators, up to few exceptions like~\cite{chensymmetry}. Inequality~\eqref{fCKNopt} holds not only for the space $C^\infty_0(\R^n)$ of all smooth functions with compact support but also for the much larger space of functions obtained by completion of $C^\infty_0(\R^n)$ with respect to the norm defined by $\|f\|^2:=\|f\|_{\dot{\mathrm H}^{s/2}(\R^n)}^2+\nrmrs f{2,2s}^2$.

\section{Subcritical interpolation inequalities}\label{Sec:Interpolation}

In this section, our purpose is to prove Theorem~\ref{Thm:Main1}.

\subsection{A Poincar\'e inequality}\label{SubSec:Poincare}

We start by recalling some basic facts:
\begin{enumerate}
\item[(i)] If $q$ and $q'$ are H\"older conjugates, then $n/q'=n-x$ with $x=n/q$,
\item[(i)] $\gamma_0(x)=1$ for any $x>0$,
\item[(ii)] $\gamma_k(n/2)=1$ and $\delta_k(n/2)=0$ for any $k\in\N$,
\item[(iii)] $\gamma_1(x)=(n-x)/x$, $\gamma_1(n/q)=q-1$ and $\delta_1(n/q_\star)=(q_\star-2)/\kappa_{n,s}$. As a consequence, we know that the first positive eigenvalues of $\mathcal K_s$ and $\mathcal L_s$ are
\[
\lambda_1(\mathcal K_s)=\gamma_1\big(\tfrac n{q_\star}\big)=q_\star-1\quad\mbox{and}\quad\lambda_1(\mathcal L_s)=\delta_1\big(\tfrac n{q_\star}\big)=\frac{q_\star-2}{\kappa_{n,s}}=\frac{2\,s}{(n-s)\,\kappa_{n,s}}\,.
\]
\end{enumerate}
A straightforward consequence is the following sharp Poincar\'e inequality.
\begin{lem}\label{Lem:Poincare} For any $F\in\mathrm H^{s/2}(\S^n)$, we have
\[
\nrm{F-F_{(0)}}2^2\le\mathsf C_{1,s}\isd{F\,\mathcal L_s\kern0.5pt F}\quad\mbox{where}\quad F_{(0)}=\isd F\,,
\]
and $\mathsf C_{1,s}=\kappa_{n,s}/(q_\star-2)$ is the optimal constant. Any function $F=F_{(0)}+F_{(1)}$, with $F_{(1)}$ such that $\mathcal L_s\,F_{(1)}=\lambda_1(\mathcal L_s)\,F_{(1)}$, realizes the equality case.\end{lem}
\begin{proof} The proof is elementary. With the usual notations, we may write
\begin{multline*}
\isd{F\,\mathcal L_s\kern0.5pt F}=\isd{(F-F_{(0)})\,\mathcal L_s\kern0.5pt(F-F_{(0)})}=\sum_{k=1}^\infty\delta_k\big(\tfrac n{q_\star}\big)\isd{|F_{(k)}|^2}\\
\ge\delta_1\big(\tfrac n{q_\star}\big)\,\nrm{F-F_{(0)}}2^2=\lambda_1(\mathcal L_s)\,\nrm{F-F_{(0)}}2^2
\end{multline*}
because $\delta_k(n/q_\star)$ is increasing with respect to $k\in\N$.
\end{proof}
The sharp Poincar\'e constant $\mathsf C_{1,s}$ is a lower bound for $\mathsf C_{q,s}$, for any $q\in(1,q_\star]$ if $s<n$, or any $q>1$ if $s=n$. Indeed, if $q\neq2$, by testing inequality~\eqref{interpolation} with $F=1+\varepsilon\,G_1$, where $G_1$ is an eigenfunction of $\mathcal L_s$ associated with the eigenvalue $\lambda_1(\mathcal L_s)$, it is easy to see that
\begin{multline*}
\varepsilon^2\,\nrm{G_1}2^2\sim\frac{\nrm Fq^2-\nrm F2^2}{q-2}\le\mathsf C_{q,s}\isd{F\,\mathcal L_s\kern0.5pt F}\\
=\mathsf C_{q,s}\,\varepsilon^2\isd{G_1\,\mathcal L_s\kern0.5pt G_1}
\end{multline*}
as $\varepsilon\to0$, which means that,
\[
\nrm{G_1}2^2=\lambda_1(\mathcal L_s)\,\mathsf C_{q,s}\,\nrm{G_1}2^2\,,
\]
by keeping only the leading order term in $\varepsilon$.
Altogether, this proves that
\be{Ineq:LowerEstimateC}
\mathsf C_{q,s}\ge\frac1{\lambda_1(\mathcal L_s)}=\frac{\kappa_{n,s}}{q_\star-2}\,.
\ee
A similar computation, with~\eqref{interpolation} replaced by~\eqref{logsob} and $F=1+\varepsilon\,G_1$, shows that
\[
\isd{|F|^2\,\log\(\frac{|F|}{\|F\|_2}\)}\sim\mathsf C_{2,s}\,\varepsilon^2\isd{G_1\,\mathcal L_s\kern0.5pt G_1}
\]
as $\varepsilon\to0$, so that~\eqref{Ineq:LowerEstimateC} also holds if $q=2$. Hence, under the Assumptions of Theorem~\ref{Thm:Main1},~\eqref{Ineq:LowerEstimateC} holds for any $q\ge1$. In order to establish Theorem~\ref{Thm:Main1} and Corollary~\ref{corlog}, we have now to prove that~\eqref{Ineq:LowerEstimateC} is actually an equality.

\subsection{Some spectral estimates}\label{SubSec:Spectral}

Let us start with some observations on the function $\gamma_k$ in~\eqref{gamma}. Expanding its expression, we get that
\[
\gamma_k(x)=\frac{(n+k-1-x)\,(n+k-2-x)\ldots(n-x)}{(k-1+x)\,(k-2+x)\ldots x}
\]
for any $k\ge1$. Taking the logarithmic derivative, we find that
\be{gamma'}
\alpha_k(x):=-\,\frac{\gamma_k'(x)}{\gamma_k(x)}=\sum_{j=0}^{k-1}\beta_j(x)\quad\mbox{with}\quad\beta_j(x):=\frac1{n+j-x}+\frac1{j+x}
\ee
and observe that $\alpha_k$ is positive. As a consequence, $\gamma_k'<0$ on $[0,n]$ and, from the expression of $\gamma_k$, we read that $\gamma_k(n)=0$. Since $\gamma_k(n/2)=1$, we know that $\gamma_k(n/q)>1$ if and only if $q>2$. Using the fact that
\[
\frac{\gamma_k''(x)}{\gamma_k(x)}=\big(\alpha_k(x)\big)^2-\,\alpha_k'(x)=\(\frac{\gamma_k'(x)}{\gamma_k(x)}\)^2+\sum_{j=0}^{k-1}\frac{(2\,j+n)\,(n-2\,x)}{(n+j-x)^2\,(j+x)^2}\,,
\]
we have $\gamma_k''(x)\ge0$, which establishes the convexity of $\gamma_k$ on $[0,n/2]$. Moreover, we know that
\[
\gamma_k'\big(\tfrac n2\big)=-\,\alpha_k\big(\tfrac n2\big)=-\sum_{j=0}^{k-1}\frac4{n+2\,j}\,.
\]
See Figure~\ref{Fig1}. Taking these observations into account, we can state the following result.
\begin{lem}\label{Lem:MonotonicityGamma} Assume that $n\ge1$. With the above notations, the function
\[
q\mapsto\frac{\gamma_k\big(\tfrac nq\big)-1}{q-2}
\]
is strictly monotone increasing on $(1,\infty)$ for any $k\ge2$.\end{lem}

\begin{figure}[ht]
\begin{center}
\begin{tikzpicture}
{\small
\draw [->] (0,0) -- (0,3.15) node {};
\draw [->] (0,0) -- (3.2,0) node [right] {$x$};
\draw [fill] (3,0) node [below] {$n$};
\draw [fill] (-0.15,0) node [below] {$0$};
\draw [fill] (1.5,0) node [below] {$\frac n2$};
\draw [fill] (1,0) node [below] {$1$};
\draw [fill] (2,0) node [below] {$2$};
\draw [thick] (1,-0.05) -- (1,0.05);
\draw [thick] (2,-0.05) -- (2,0.05);
\draw [thick] (1.5,-0.05) -- (1.5,0.05);
\draw [thick] (3,-0.05) -- (3,0.05);
\draw [fill] (1.5,0.5) circle [radius=0.03];
\draw [fill] (2,0.5) circle [radius=0.03];
\draw [fill] (-0.15,0.5) node {$1$};
\draw [fill] (3.2,0.5) node {$\gamma_0$};
\draw [dotted, domain=0:3] plot({\x},{0.5});
\draw [dotted, domain=0.425:3] plot({\x},{1.5/\x-0.5});
\draw [thick, domain=0.86:3] plot({\x},{0.5*(3-\x)*(4-\x)*(5-\x)/\x/(\x+1)/(\x+2)});
\draw [fill] (1.1,3) node {$\gamma_k$};
\draw [fill] (2,3) node[right] {$k\ge 2$};
\draw [fill] (0.55,3) node[left] {$\gamma_1$};
\draw [dotted, domain=1:3] plot({\x},{1.5/3*\x-0.5});
\draw [thick, domain=1:3] plot({3/\x},{0.5*(3-\x)*(4-\x)*(5-\x)/\x/(\x+1)/(\x+2)});
\draw [fill] (3.2,2) node[above] {$\gamma_k(\frac n\cdot)$};
\draw [fill] (3.2,0.9) node[above] {$\gamma_1(\frac n\cdot)$};
}
\end{tikzpicture}
\caption{\label{Fig1}{\sl\small The functions $x\mapsto\gamma_k(x)$ and $q\mapsto\gamma_k(n/q)$ are both convex, and such that $\gamma_k(n/2)=1$.}}
\end{center}
\end{figure}
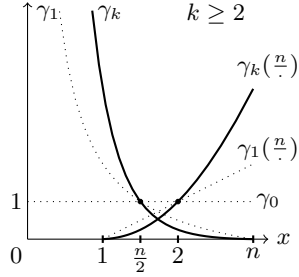

\begin{proof} Let us prove that $q\mapsto\gamma_k(n/q)$ is strictly \emph{convex} with respect to~$q$ for any $k\ge 2$. Written in terms of $x=n/q$, it is sufficient to prove that
\[
x\,\gamma_k''+ 2\,\gamma_k'>0\quad\forall\,x\in(0,n)\,,
\]
which can also be rewritten as
\[
\alpha_k^2-\alpha_k'-\tfrac2x\,\alpha_k>0\,.
\]
Let us prove this inequality. Using the estimates
\[
\alpha_k^2=\(\sum_{j=0}^{k-1}\beta_j\)^2\ge2\,\beta_0\sum_{j=1}^{k-1}\beta_j+\sum_{j=0}^{k-1}\beta_j^2\,,
\]
\[
\beta_0^2-\beta_0'-\tfrac2x\,\beta_0=0\,,
\]
and
\[
2\,\beta_0\,\beta_j+\beta_j^2-\beta_j'-\tfrac2x\,\beta_j=\frac{2\,(n+j)\,(n+2\,j)}{(n-x)\,(n+j-x)\,(j+x)^2}
\]
for any $j\ge1$, we actually find that
\[
\alpha_k^2-\alpha_k'-\tfrac2x\,\alpha_k\ge\sum_{j=1}^{k-1}\frac{2\,(n+j)\,(n+2\,j)}{(n-x)\,(j+n-x)\,(j+x)^2}\quad\forall\,k\ge2\,,
\]
which concludes the proof. Note that as a byproduct, we also proved the strict convexity of $\gamma_k$ for the whole range $x\in(0,n)$. See Figure 2 for a summary of properties of the spectral functions. \end{proof}

\begin{proof}[Proof of Theorem~\ref{Thm:Main1}] We deduce from~\eqref{Sobolev} that
\[
\frac{\nrm Fq^2-\nrm F2^2}{q-2}\le\sum_{k=1}^\infty\frac{\gamma_k\big(\tfrac nq\big)-1}{q-2}\isd{|F_{(k)}|^2}
\]
because $\gamma_0(x)=1$. It follows from Lemma~\ref{Lem:MonotonicityGamma} that
\begin{multline*}
\frac{\nrm Fq^2-\nrm F2^2}{q-2}\le\sum_{k=1}^\infty\frac{\gamma_k\big(\tfrac n{q_\star}\big)-1}{q_\star-2}\isd{|F_{(k)}|^2}\\=\frac{\kappa_{n,s}}{q_\star-2}\sum_{k=1}^\infty\delta_k\big(\tfrac n{q_\star}\big)\isd{|F_{(k)}|^2}=\frac{\kappa_{n,s}}{q_\star-2}\isd{F\,\mathcal L_s\kern0.5pt F}\,.
\end{multline*}
This proves that $\mathsf C_{q,s}\le\frac{\kappa_{n,s}}{q_\star-2}$. The reverse inequality has already been shown in~\eqref{Ineq:LowerEstimateC}.\end{proof}

\begin{proof}[Proof of Theorem~\ref{Thm:Main1bis}] With $s\in(-n,0)$, it turns out that $q_\star$ defined by~\eqref{CriticalExponent} is in the range $(1,2)$ and plays the role of $p$ in~\eqref{HLS3}. According to Lemma~\ref{Lem:MonotonicityGamma}, the inequality holds with the same constant for any $q\in(1,q_\star)$, and this constant is optimal because of~\eqref{Ineq:LowerEstimateC}.\end{proof}

\subsection{An improved inequality with a remainder term}\label{SubSec:Remainder}

What we have shown in Section~\ref{SubSec:Spectral} is actually that the fractional Sobolev inequality~\eqref{Sobolev} is equivalent to the following improved subcritical inequality.
\begin{cor}\label{Cor:improvedInterpolation} Assume that $n\ge1$, $q\in[1,2)\cup(2,q_\star)$ if $s\in(0,n)$, and $q\in[1,2)\cup(2,\infty)$ if $s=n$. For any $F\in\mathrm H^{s/2}(\S^n)$ we have
\[
\frac{\nrm Fq^2-\nrm F2^2}{q-2}+\isd{F\,\mathcal R_{q, s}\kern0.5pt F}\le\frac{\kappa_{n,s}}{q_\star-2}\isd{F\,\mathcal L_s\kern0.5pt F}
\]
where $\mathcal R_{q,s}$ is a positive semi-definite operator whose kernel is generated by the spherical harmonics corresponding to $k=0$ and $k=1$.\end{cor}
\begin{proof} We observe that
\[
\isd{F\,\mathcal R_{q,s}\kern0.5pt F}:=\sum_{k=2}^\infty\epsilon_k\isd{|F_{(k)}|^2}
\]
where
\[
\epsilon_k:=\frac{\gamma_k\big(\tfrac n{q_\star}\big)-1}{q_\star-2}-\frac{\gamma_k\big(\tfrac nq\big)-1}{q-2}
\]
is positive for any $k\ge2$ according to Lemma~\ref{Lem:MonotonicityGamma}.\end{proof}

Equality in~\eqref{interpolation} is realized only when $F$ optimizes the critical fractional Sobolev inequality and, if~$q<q_\star$, when $F_{(k)}=0$ for any $k\ge2$, which is impossible unless $F$ is an optimal function for the Poincar\'e inequality of Lemma~\ref{Lem:Poincare}. This observation will be further exploited in Section~\ref{Sec:Conclusion}.

\subsection{Fractional logarithmic Sobolev inequalities}\label{Sec:LogSob}

\begin{proof}[Proof of Corollary~\ref{corlog}] According to Theorem~\ref{Thm:Main1}, we know by~\eqref{interpolation} that
\[
\frac{\nrm Fq^2-\nrm F2^2}{q-2}\le\frac{n-s}{2\,s}\,\kappa_{n,s}\isd{F\,\mathcal L_s\kern0.5pt F}
\]
for any function $F\in\mathrm H^{s/2}(\S^n)$ and any $q\in[1,2)\cup(2,q_\star)$ with $q_\star=q_\star(s)$ given by~\eqref{CriticalExponent} (and the convention that $q_\star=\infty$ if $s=n$). Taking the limit as $q\to2$ for a given $s\in(0,n)$, we obtain that~\eqref{logsob} holds with $\mathsf C_{2,s}\le\frac{n-s}{2\,s}\,\kappa_{n,s}$. The reverse inequality has already been shown in~\eqref{Ineq:LowerEstimateC} written with $q=2$.\end{proof}

Let us comment on the results of Corollary~\ref{corlog}, in preparation for Section~\ref{Sec:Conclusion}. Instead of fixing $s$ and letting $q\to2$ as in the proof of Corollary~\ref{corlog}, we can consider the case $q=q_\star(s)$ and let $s\to0$, or equivalently rewrite~\eqref{Sobolev} as
\[
\frac{\nrm Fq^2-\nrm F2^2}{q-2}\le\sum_{k=0}^\infty\frac{\gamma_k\big(\tfrac nq\big)-1}{q-2}\isd{|F_{(k)}|^2}\,,
\]
and take the limit as $q\to2$. By an endpoint differentiation argument, we recover the conformally invariant \emph{fractional logarithmic Sobolev} inequality
\be{LogSob}
\isd{F^2\,\log\(\frac{|F|}{\nrm F2}\)}\le\frac n2\isd{F\,\mathcal K_0'\kern0.5pt F}
\ee
as in~\cite{MR1164616,MR1441924}, where the differential operator $\mathcal K_0'$ is the endpoint derivative of $\mathcal K_s$ at $s=0$. The equality $\mathcal K_0'=\mathcal L_0'$ holds because $\kappa_{n,0}=1$ and $\mathcal K_0=\mathrm{Id}$. More specifically the right-hand side of~\eqref{LogSob} can be written using the identities
\[
\isd{F\,\mathcal K_0'\kern0.5pt F}=\isd{F\,\mathcal L_0'\kern0.5pt F}
=\frac12\,\sum_{k=0}^\infty \alpha_k\big(\tfrac n2\big) \isd{|F_{(k)}|^2}
\]
with
\[
\alpha_k\big(\tfrac n2\big)=-\,\gamma_k'\big(\tfrac n2\big)=\sum_{j=0}^{k-1}\frac4{n+2\,j}\,.
\]

Inequality~\eqref{LogSob} is sharp, and equality holds if and only if $F$ is obtained by applying any conformal transformation on $\S^n$ to constant functions. Finally, let us notice that~\eqref{LogSob} can be recovered as an endpoint of~\eqref{logsob} by letting $s\to0$. The critical case is then achieved as a limit of the subcritical inequalities~\eqref{logsob}. The optimal constant can be identified, but the set of optimal functions in the limit is larger than in the subcritical regime, because of the conformal invariance.

Even more interesting is the fact that the \emph{fractional logarithmic Sobolev} inequality is critical for $s=0$ and $q=2$ but subcritical inequalities corresponding to $q\in[1,2)$ still make sense.
\begin{cor}\label{Cor:s=0} Assume that $n\ge1$ and $q\in[1,2)$. For any $F\in\mathrm L^2(\S^n)$ such that $\isd{F\,\mathcal K_0'\kern0.5pt F}$ is finite, we have
\[
\frac{\nrm Fq^2-\nrm F2^2}{q-2}\le\frac n2\isd{F\,\mathcal K_0'\kern0.5pt F}\,.
\]\end{cor}
As for Corollary~\ref{corlog}, the proof relies on Lemma~\ref{Lem:MonotonicityGamma}. Details are left to the reader.

\section{Stereographic projection and weighted fractional interpolation inequalities on the Euclidean space}\label{Sec:Euclidean}

This section is devoted to the proof of Theorem~\ref{Thm:Main2}. Various results concerning the extension of the Caffarelli-Kohn-Nirenberg inequalities introduced in~\cite{Caffarelli-Kohn-Nirenberg-84} (see also~\cite[Theorem~1]{delPino20102045} in our context) are scattered throughout the literature, and one can consult for instance ~\cite[Theorem 1.8]{MR2869807} for a quite general result in this direction. However, very little is known so far on optimal constants or even estimates of such constants, except for some limit cases like fractional Sobolev or fractional Hardy-Sobolev inequalities (see, \emph{e.g.}, \cite{MR3334182}). What we prove here is that the interpolation inequalities on the sphere provide inequalities on the Euclidean space with weights based on $(1+|x|^2)$, with optimal constants.

\begin{proof}[Proof of Theorem~\ref{Thm:Main2}] Let us consider the stereographic projection $\mathcal S$, whose inverse is defined by
\[
\mathcal S^{\kern0.5pt-1}:\R^n\rightarrow\S^n\,,\quad x\longmapsto\zeta=\(\frac{2\,x}{1+|x|^2},\,\frac{1-|x|^2}{1+|x|^2}\)\,.
\]
with Jacobian determinant $|J|=2^n\,(1+|x|^2)^{-n}$. Given $s\in(0,n)$ and $q\in(2,q_\star)$, and using the conformal Laplacian, we can write inequality~\eqref{interpolation} as
\[
\nrm Fq^2-\frac{q_\star-q}{q_\star-2}\,\nrm F2^2\le\frac{q-2}{q_\star-2}\,\kappa_{n,s}\isd{F\,\mathcal A_s\kern0.5pt F}
\]
where $\mathcal A_s$ and the fractional Laplacian on $\R^n$ are related by
\[
|J|^{1-\frac1{q_\star}}\,(\mathcal A_s\kern0.5ptF)\circ\mathcal S^{\kern0.5pt-1}=(-\Delta)^{s/2}\,\Big(|J|^\frac1{q_\star}\,F\circ\mathcal S^{\kern0.5pt-1}\Big)\,.
\]
Then the interpolation inequality \eqref{interpolation} on the sphere is equivalent to the following \emph{fractional interpolation inequality on the Euclidean space}
\begin{multline*}
|\S^n|^{1-\frac 2q}\(\ird{|f|^q\,|J|^{1-\frac q{q_\star}}}\)^{\frac2q}-\frac{q_\star-q}{q_\star-2}\ird{f^2\,|J|^{1-\frac2{q_\star}}}\\
\le\frac{q-2}{q_\star-2}\,\kappa_{n,s}\ird{f\,(-\Delta)^{s/2}\kern0.5ptf}
\end{multline*}
by using the change of variables $F\longmapsto f=|J|^{1/q_\star}\,F\circ\mathcal S^{\kern0.5pt-1}$. The equality case is now achieved only by $f=|J|^{1/q_\star}$ for any $q\in(2,q_\star)$, up to a multiplication by a constant, and the inequality is equivalent to~\eqref{fCKNopt}.\end{proof}

\section{Concluding remarks}\label{Sec:Conclusion}

A striking feature of inequality~\eqref{interpolation} is that the optimal constant $\mathsf C_{q,s}$ is determined by a linear eigenvalue problem, although the problem is definitely nonlinear. This deserves some comments. Let $q\in[1,2)\cup(2,q_\star)$ if $s<n$ and $q\in[1,2)\cup(2,\infty)$ if $s=n$. With $\mathcal Q$ defined by~\eqref{Q} on $\mathscr H^{s/2}$, the subset of the functions in $\mathrm H^{s/2}(\S^n)$ which are not almost everywhere constant, we investigate the relation
\[
\mathsf C_{q,s}\inf_{F\in\mathscr H^{s/2}}\mathcal Q[F]=1\,.
\]
Notice that both numerator and denominator of $\mathcal Q[F]$ converge to $0$ if $F$ approaches a constant, so that $\mathcal Q$ becomes undetermined in the limit. As we shall see next, this happens for a minimizing sequence and explains why a linearized problem appears in the limit.

By compactness of the Sobolev embedding $\mathrm H^{s/2}(\S^n)\hookrightarrow\mathrm L^q(\S^n)$ (see \cite{MR0450957,MR2869807} for fundamental properties of fractional Sobolev spaces, \cite[sections~6 and~7]{MR2944369} and~\cite{MR3445279} for application to variational problems), any minimizing sequence $(F_n)_{n\in\N}$ for $\mathcal Q$ is relatively compact if we assume that $\nrm{F_n}q=1$ for any $n\in\N$. This normalization can be imposed without loss of generality because of the homogeneity of $\mathcal Q$. Hence $(F_n)_{n\in\N}$ converges to a limit $F\in\mathrm H^{s/2}(\S^n)$. Assume that $F$ is not a constant. Then the denominator in $\mathcal Q[F]$ is positive and by semicontinuity we know that
\[
\isd{F\,\mathcal L_s\kern0.5pt F}\le\lim_{n\to+\infty}\isd{F_n\,\mathcal L_s\kern0.5pt F_n}\,.
\]
On the other hand, by compactness, up to the extraction of a subsequence, we have that
\[
\nrm F2^2=\lim_{n\to+\infty}\nrm{F_n}2^2\quad\mbox{and}\quad\nrm Fq^2=\lim_{n\to+\infty}\nrm{F_n}q^2=1\,.
\]
Hence $F$ is optimal and solves the Euler-Lagrange equations
\[
(q-2)\,\mathsf C_{q,s}\,\mathcal L_s\kern0.5pt F+F=F^{q-1}\,.
\]
Using Corollary~\ref{Cor:improvedInterpolation}, we also get that $F$ lies in the kernel of $\mathcal R_{q,s}$, that is, the space generated by the spherical harmonics corresponding to $k=0$ and $k=1$. From the Euler-Lagrange equations, we read that $F$ has to be a constant. Because of the normalization $\nrm Fq=1$, we obtain that $F=1$ a.e., a contradiction.

Hence $(F_n)_{n\in\N}$ converges to $1$ in $\mathrm H^{s/2}(\S^n)$. With $\varepsilon_n=\|1-F_n\|_{\mathrm H^{s/2}(\S^n)}$ and $v_n:=(F_n-1)/\varepsilon_n$, we can write that
\[
F_n=1+\varepsilon_n\,v_n\quad\mbox{with}\quad\|v_n\|_{\mathrm H^{s/2}(\S^n)}=1\quad\forall\,n\in\N
\]
and
\[
\lim_{n\to+\infty}\varepsilon_n=0\,.
\]
On the other hand, $(F_n)_{n\in\N}$ being a minimizing sequence, it turns out that
\[
\mathsf C_{q,s}^{-1}=\lim_{n\to+\infty}\mathcal Q[F_n]=\lim_{n\to+\infty}\frac{\varepsilon_n^2\,(q-2)\isd{v_n\,\mathcal L_s\kern0.5pt v_n}}{\nrm{1+\varepsilon_n\,v_n}q^2-\nrm {1+\varepsilon_n\,v_n}2^2}\,.
\]
If $q>2$, an elementary computation shows that
\be{Taylor}
\nrm{1+\varepsilon_n\,v_n}q^2-\nrm {1+\varepsilon_n\,v_n}2^2=(q-2)\,\varepsilon_n^2\,\nrm{v_n-\bar v_n}2^2(1+o(1))
\ee
as $n\to+\infty$, where $\bar v_n:=\isd{v_n}$, so that
\[
\mathsf C_{q,s}^{-1}=\lim_{n\to+\infty}\mathcal Q[F_n]=\lim_{n\to+\infty}\frac{\isd{v_n\,\mathcal L_s\kern0.5pt v_n}}{\nrm{v_n-\bar v_n}2^2}\,.
\]
Details on the Taylor expansion used in~\eqref{Taylor} can be found in Appendix~\ref{Appendix:Taylor}. When $q\in[1,2)$, we can estimate the denominator by restricting the integrals to $\{x\in\S^n\,:\,\varepsilon_n\,|v_n|<1/2\}$ and Taylor expand $t\mapsto(1+t)^q$ on $(1/2,3/2)$.

Notice that by $F_n$ being a function in $\mathscr H^{s/2}$, we know that $\nrm{v_n-\bar v_n}2>0$ for any $n\in\N$, so that the above limit makes sense. With the notations of Section~\eqref{SubSec:Poincare}, we know that
\[
\mathsf C_{q,s}^{-1}\ge\inf_{v\in\mathscr H^{s/2}}\frac{\isd{v\,\mathcal L_s\kern0.5pt v}}{\nrm{v-\bar v}2^2}\ge\lambda_1(\mathcal L_s)=\frac{2\,s\,\kappa_{n,s}}{n-s}
\]
according to the Poincar\'e inequality of Lemma~\ref{Lem:Poincare}, which proves that we actually have equality in~\eqref{Ineq:LowerEstimateC} and determines $\mathsf C_{q,s}$.

Additionally, we may notice that $(v_n)_{n\in\N}$ has to be a minimizing sequence for the Poincar\'e inequality, which means that up to a normalization and after the extraction of a subsequence, $v_n-\bar v_n$ converges to a spherical harmonic function associated with the component corresponding to $k=1$. This explains why we obtain that $\mathsf C_{q,s}\,\lambda_1(\mathcal L_s)=1$.

The above considerations have been limited to the subcritical range $q<q_\star$ if $s<n$ and $q<+\infty$ if $s=n$. However, the critical case of the Sobolev inequality can be obtained by passing to the limit as $q\to q_\star$ (and even the Onofri type inequalities when $s=n$) so that the optimal constants are also given by an eigenvalue in the critical case. However, due to the conformal invariance, the constant function $F\equiv1$ is not the only optimal function. At this point it should be noted that the above considerations heavily rely on Corollary~\ref{Cor:improvedInterpolation} and, as a consequence, cannot be used to give a variational proof of Theorem~\ref{Thm:Main1}.

\medskip Although the subcritical interpolation inequalities of this paper appear weaker than inequalities corresponding to a critical exponent, we are able to identify the equality cases and the optimal constants. We are also able to keep track of a remainder term which characterizes the functions realizing the optimality of the constant or, to be precise, the limit of any minimizing sequence and its first order correction. This first order correction, or equivalently the asymptotic value of the quotient $\mathcal Q$, determines the optimal constant and explains the role played by the eigenvalues in a problem which is definitely nonlinear.

\appendix\section{The spectrum of the fractional Laplacian}\label{Appendix:Spectrum}

The standard approach for computing $\gamma_k$ in~\eqref{gamma} relies on the Funk-Hecke formula as it is detailed in~\cite[Section~4]{MR2848628}. In this appendix, for completeness, we provide a simple, direct proof of the expression of $\gamma_k$. For this purpose, we compute the eigenvalues $\lambda_k=\lambda_k\big((-\Delta)^{s/2}\big)$ of the fractional Laplacian on $\R^n$, that is,
\[
(-\Delta)^{s/2}\kern0.5pt f_k=\frac{\lambda_k}{(1+|x|^2)^s}\,f_k\quad\mbox{in}\quad\R^n\,,
\]
for any $k\in\N$. We shall then deduce the eigenvalues of $\mathcal L_s$. This determines the optimal constant in~\eqref{Sobolev} and~\eqref{interpolation} without using Lieb's duality and without relying on the symmetry of the optimal case in~\eqref{HLS} as in~\cite{MR717827}.
\begin{prop}\label{Prop:Spectrum} Given $s\in(0,n)$, the spectrum of the fractional Laplacian is
\[
\lambda_k\big((-\Delta)^{s/2}\big)=2^s\,\frac{\Gamma(k+\frac n{q'})}{\Gamma(k+\frac nq)}=2^s\,\lambda_k(\mathcal A_s)=2^s\,\frac{\Gamma(\frac n{q'})}{\Gamma(\frac nq)}\,\lambda_k(\mathcal K_s)\,.
\]
\end{prop}
\begin{proof} Using the stereographic projection and a decomposition in spherical harmonics, we can reduce the problem of computing the spectrum to the computation of the spectrum associated with the eigenfunctions
\[
f_k^\mu(x)=C_k^{(\alpha)}(z)\,(1+|x|^2)^{-\mu}\quad\mbox{with}\quad z=\tfrac{1-|x|^2}{1+|x|^2}\,,
\]
where $\mu=\lambda/2=(n-s)/2$, $\alpha=(n-1)/2$ and $C_k^{(\alpha)}$ denotes the Gegenbauer polynomials. Let $\hat f(\xi)=(\mathscr Ff)(\xi):=\ird{f(x)\,e^{-\,2\,\pi\,i\,\xi\cdot x}}$ be the Fourier transform of a function $f$. Since the functions are radial, by the Hankel transform $\mathcal H^{\frac n2-1}$, we get that
\[
\widehat{f_k^\mu}(\xi)=\frac{2\,\pi}{|\xi|^{\frac n2-1}}\int_0^\infty f_k^\mu(r)\,J_{\frac n2-1}\big(2\,\pi\,r\,|\xi|\big)\,r^{\frac n2}\,dr
\]
(\emph{cf.}~\cite[Appendix B.5, p.~578]{MR3243734}) where $J_\nu$ is the Bessel function of the first kind.

The Fourier transform of $f_0^\mu=(1+|x|^2)^{-\mu}$ has been calculated, \emph{e.g.}, by E.~Lieb in~\cite[(3.9)-(3.14)]{MR717827} in terms of the modified Bessel functions of the second kind $K_\nu$ as
\[
\widehat{f_0^\mu}(\xi)=\frac{\pi^{\frac n2}\,2^{1+\frac2n-\mu}}{\Gamma(\mu)}\,\big(2\,\pi\,|\xi|\big)^{\mu-\frac n2}\,K_{\mu-\frac n2}\big(2\,\pi\,|\xi|\big)\,.
\]
This is a special case of the modified Weber-Schafheitlin integral formula in~\cite[Chapter XIII, Section 13.45, p.~ 410]{MR1349110}. Using the expansion of Gegenbauer polynomials, we get
\begin{eqnarray*}
\widehat{f_k^\mu}(\xi)&=&\frac{2\,\pi}{\Gamma(\frac{n-1}2)\,|\xi|^{\frac n2-1}}\sum_{j=0}^{[\frac k2]}\sum_{l=0}^{k-2\,j}\Bigg[\frac{(-1)^{j+k-l}\,2^{k+l-2\,j}\,\Gamma(\frac{n-1}2+k-j)}{j!\,k!\,(k-2\,j-l)!}\\
&&\hspace*{4cm}\times\int_0^\infty(1+r^2)^{-(\mu+l)}\,J_{\frac n2-1}\big(2\,\pi\,r\,|\xi|\big)\,r^{\frac n2}\,dr\Bigg]\\
&=&\frac1{\Gamma(\frac{n-1}2)}\sum_{j=0}^{[\frac k2]}\sum_{l=0}^{k-2\,j}\frac{(-1)^{j+k-l}\,2^{k+l-2\,j}\,\Gamma(\frac{n-1}2+k-j)}{j!\,l!\,(k-2\,j-l)!}\,\widehat{f_0^{\mu+l}}(\xi)\\
&=&\frac{2^{1+\frac2n-\mu}\,\pi^{\frac n2}\big(2\,\pi\,|\xi|\big)^{\mu-\frac n2}}{\Gamma(\frac{n-1}2)\,\Gamma(\mu+k)}\,I_{n,k}^\mu(|\xi|)\,,
\end{eqnarray*}
where
\begin{eqnarray*}
I_{n,k}^\mu(|\xi|)&:=&\sum_{l=0}^kc_{n,k,l}\,\frac{\Gamma(\mu+k)}{\Gamma(\mu+l)}\,\big(2\,\pi\,|\xi|\big)^l\,K_{\mu-\frac n2+l}\big(2\,\pi\,|\xi|\big)\,,\\
\text{and}\qquad c_{n,k,l}&:=&\frac1{l!}\sum_{j=0}^{[\frac{k-l}2]}\,\frac{(-1)^{j+k-l}\,2^{k-2\,j}\,\Gamma(\frac{n-1}2+k-j)}{j!\,(k-2\,j-l)!}\,.
\end{eqnarray*}
{}From the recurrence relation
\[
x\,(K_{\nu-1}-K_{\nu+1})=-\,2\,\nu\,K_\nu\,,
\]
we deduce the identity
\[
\sum_{l=0}^k\,c_{n,k,l}\,x^l\,\(\frac{\Gamma(\nu+\frac n2+k)}{\Gamma(\nu+\frac n2+l)}\,K_{\nu+l}(x)-\frac{\Gamma(-\nu+\frac n2+k)}{\Gamma(-\nu+\frac n2+l)}\,K_{\nu-l}(x)\)=0\quad\forall\,k\ge0
\]
and observe that
\[
I^{\mu_1}_{n,k}=I^{\mu_2}_{n,k}\quad\forall\,k\in\N
\]
if $\mu_1=\lambda/2$ and $\mu_2=\lambda/2+s$, so that $\mu_1+\mu_2=n$ and $\mu_1-\mu_2=-\,s$. It remains to observe that
\[
(2\,\pi\,|\xi|)^s\,\widehat{f_k^{\lambda/2}}=\lambda_k\,\mathscr F\(f_k^{\lambda/2}\,(1+|x|^2)^{-s}\)\quad\mbox{with}\quad\lambda_k=2^s\,\frac{\Gamma(k+\frac n{q'})}{\Gamma(k+\frac nq)}\,.
\]
\end{proof}

\section{A Taylor formula with integral remainder term}\label{Appendix:Taylor}

Let us define the function $r:\R\to\R$ such that
\[
|1+t|^q=1+q\,t+\frac12\,q\,(q-1)\,t^2+r(t)\quad\forall\,t\in\R\,.
\]
\begin{lem}\label{Lem:Taylor} Let $q\in(2,\infty)$. With the above notations, there exists a constant $C>0$ such that
\[|r(t)| \le \, \left\{\begin{array}{cr}
C\, |t|^3 & \text{if}~\, |t|\le 1 \\
C\, |t|^q & \text{if}~\, |t|\ge 1 \\
\end{array}\right.\]
\end{lem}
This result is elementary but crucial for the expansion of $\nrm Fq^2-\nrm F2^2$ around $F=1$. This is why we give a proof with some details, although we claim absolutely no originality for that. Similar computations have been repeatedly used in a related context, \emph{e.g.}, in~\cite{MR3179693,christ2014sharpened,MR3429269}.

\begin{proof} Using the Taylor formula with integral remainder term
\[
f(t)=f(0)+f'(0)\,t+\frac12\,f''(0)\,t^2+\frac12\int_0^t(t-s)^2\,f'''(s)\,ds
\]
applied to $f(t)=(1+t)^q$ with $q>2$, we obtain that
\[
|1+t|^q=1+q\,t+\frac12\,q\,(q-1)\,t^2+r(t)
\]
where the remainder term is given by
\[
r(t)=\frac12\,q\,(q-1)\,(q-2)\,t^q\int_0^1(1-\sigma)^2\,\left|\frac1t+\sigma\right|^{q-4}\(\frac1t+\sigma\right)\,d\sigma\,.
\]
Hence the remainder term can be bounded as follows:
\begin{enumerate}
\item[(i)] if $t\ge1$, using $\sigma<\frac1t+\sigma<1+\sigma$, we get that
\[
0<r(t)<c_q\,t^q
\]
with $c_q=\frac12\,q\,(q-1)\,(q-2)\int_0^1(1-\sigma)^2\,\max\{\sigma^{q-3},(1+\sigma)^{q-3}\}\,d\sigma$.

\item[(ii)] if $0<t<1$, using $\frac1t<\frac1t+\sigma<\frac2t$, we get that
\[
0<r(t)<\frac16\,q\,(q-1)\,(q-2)\,\max\{1,2^{q-3}\}\,t^3\,.
\]

\item[(iii)] if $-1<t<0$, using $\frac1t<\frac1t+\sigma<\frac1t+1<0$, we get that
\[
-\frac16\,q\,(q-1)\,(q-2)\,|t|^3<r(t)<0\,.
\]

\item[(iv)] if $t\le-1$, using $\sigma-1<\frac1t+\sigma<\sigma$, we get that
\[
-\frac12\,(q-1)\,(q-2)\,t^q<r(t)<t^q\,.
\]
\end{enumerate}
\end{proof}

\section{Notations and ranges}\label{Appendix:Notations}

For the convenience of the reader, this appendix collects various notations which are used throughout this paper and summarizes the ranges covered by the parameters.

The identity
\[
\lambda=\frac{2\,n}{p'}\quad\mbox{where}\quad\frac1p+\frac1{p'}=1
\]
means that
\[
p=\frac{2\,n}{2\,n-\lambda}\,.
\]
With
\[
\lambda=n-s\,,
\]
we have
\[
p=\frac{2\,n}{n+s}\quad\mbox{and}\quad p'=q_\star=\frac{2\,n}{n-s}\,.
\]
The limiting values of the parameters are summarized in Table~\ref{Table:1}.
\begin{table}[ht]
\begin{tabular}{|c|ccc|}\hline
$s$&$0$&$2$&$n$\\\hline
$\lambda$&$n$&$n-2$&$0$\\\hline
$p$&$2$&$\frac{2\,n}{n+2}$&$1$\\\hline
$p'=q_\star$&$2$&$\frac{2\,n}{n-2}$&$+\infty$\\\hline
\end{tabular}\vspace*{6pt}
\caption{\sl\label{Table:1}Correspondence of the limiting values of the parameters.}
\end{table}

The coefficients $\gamma_k$ and $\delta_k$ defined by
\begin{multline*}
\gamma_k(x)=\frac{\Gamma(x)\,\Gamma(n-x+k)}{\Gamma(n-x)\,\Gamma(x+k)}\\
\mbox{and}\quad\delta_k(x)=\frac1{\kappa_{n,s}}\,\big(\gamma_k(x)-1\big)=\frac{\Gamma(n-x+k)}{\Gamma(x+k)}-\frac{\Gamma(n-x)}{\Gamma(x)}
\end{multline*}
are such that
\[
\delta_k\big(\tfrac n{q_\star}\big)=\frac1{\kappa_{n,s}}\,\big(\gamma_k\big(\tfrac n{q_\star}\big)-1\big)\quad\mbox{where}\quad\kappa_{n,s}=\frac{\Gamma\big(\frac n{q_\star}\big)}{\Gamma\big(n-\frac n{q_\star}\big)}=\frac{\Gamma\big(\frac{n-s}2\big)}{\Gamma\big(\frac{n+s}2\big)}\,.
\]

We recall that $\gamma_0(n/q)-1=0$, $\gamma_1(n/q)-1=q-2$, $\delta_k\big(\tfrac n{q_\star}\big)=k\,(k+n-1)$ and $1/\kappa_{n,2}=\frac14\,n\,(n-2)$. According to~\eqref{gamma'}, we have that
\[
\alpha_k(x)=-\,\frac{\gamma_k'(x)}{\gamma_k(x)}=\sum_{j=0}^{k-1}\beta_j(x)\quad\mbox{with}\quad\beta_j(x)=\frac1{n+j-x}+\frac1{j+x}
\]
for any $k\ge1$. With these notations, the eigenvalues of $\mathcal K_s$, $\mathcal L_s$ and $\mathcal K_0'=\mathcal L_0'$ are respectively given by
\[
\gamma_k\big(\tfrac n{q_\star(s)}\big)=\gamma_k\big(\tfrac{n-s}2\big), \quad \frac 1{\kappa_{n,s}}\,\big(\gamma_k\big(\tfrac{n-s}2\big)-1\big), \quad \frac12\,\alpha_k\big(\tfrac n2\big)
\]
with
\[
\alpha_k\big(\tfrac n2\big)=-\,\gamma_k'\big(\tfrac n2\big)=4\sum_{j=0}^{k-1}\frac 1{n+2\,j}\,.
\]

Finally, we recall that $\mathcal K_s$, the fractional Laplacian $\mathcal L_s$ and the conformal fractional Laplacian $\mathcal A_s$ satisfy the relations
\[
\kappa_{n,s}\,\mathcal A_s=\mathcal K_s=\kappa_{n,s}\,\mathcal L_s+\,\mathrm{Id}\,.
\]

\par\medskip\noindent{\small{\bf Acknowledgements.} This work is supported by a public grant overseen by the French National Research Agency (ANR) as part of the ``Investissements d'Avenir'' program (A.Z., reference: ANR-10-LABX-0098, LabEx SMP) and by the projects \emph{STAB} (J.D., A.Z.) and \emph{Kibord} (J.D.) of the French National Research Agency (ANR). A.Z.~thanks the ERC Advanced Grant BLOWDISOL (Blow-up, dispersion and solitons; PI: Frank Merle) \# 291214 for support. The authors thank Maria J.~Esteban for fruitful discussions and suggestions, and are also grateful to Van Hoang Nguyen for pointing an error in a former version.
\par\smallskip\noindent\copyright~2016 by the authors. This paper may be reproduced, in its entirety, for non-commercial purposes.}

\end{document}